\documentclass{amsart}

\usepackage{amsmath}
\usepackage{amsthm}
\usepackage{amssymb}

\allowdisplaybreaks

\numberwithin{equation}{section}

\newcommand{\Z}{\mathbb{Z}}
\newcommand{\N}{\mathbb{N}}
\newcommand{\R}{\mathbb{R}}
\newcommand{\F}{\mathcal{F}}
\newcommand{\Sh}{\mathcal{S}}

\newcommand{\supp}{\mathop{\mathrm{supp}}}

\newcommand{\la}{\langle}
\newcommand{\ra}{\rangle}
\newcommand{\K}{\boldsymbol{k}}
\newcommand{\Nu}{\boldsymbol{\nu}}

\theoremstyle{plain}
\newtheorem{thm}{Theorem}[section]
\newtheorem{prop}[thm]{Proposition}
\newtheorem{lem}[thm]{Lemma}

\newtheorem*{thmA}{Theorem A}

\theoremstyle{definition}

\newtheorem{rem}[thm]{Remark}

\begin{document}
%=============================================================================================
%=============================================================================================
\title[Bilinear pseudo-differential operators of $S_{0,0}$-type]
{Kato-Ponce type inequality for bilinear pseudo-differential operators 
of $S_{0, 0}$-type in the scale of Besov spaces}

\author[N. Shida]{Naoto Shida}

\date{}

\address{Department of Mathematics, 
Graduate School of Science, Osaka University, 
Toyonaka, Osaka 560-0043, Japan}

\email[N. Shida]{u331453f@ecs.osaka-u.ac.jp}

\keywords{Besov spaces, Bilinear H\"ormander symbol classes,
Bilinear pseudo-differential operators, Kato-Ponce type inequality}

\subjclass[2020]{35S05, 42B15, 42B35}
%=============================================================================================

%=============================================================================================
%=============================================================================================
\begin{abstract}
We consider the Kato-Ponce type inequality for bilinear pseudo-differential operators with $S_{0, 0}$-type 
symbols in the scale of Besov spaces. In particular, the borderline whether the boundedness of those operators 
holds or not is discussed.
\end{abstract}
%=============================================================================================
%=============================================================================================

%=============================================================================================
%=============================================================================================
\maketitle
%=============================================================================================
%=============================================================================================

%=============================================================================================
%=============================================================================================
\section{Introduction}
In this paper,  the following $S_{0, 0}$-type symbol classes are considered.
One is the bilinear H\"ormander class $BS^m_{0, 0}$, $m \in \R$, 
consisting of all $\sigma(x, \xi_1, \xi_2) \in C^\infty((\R^n)^3)$ such that
\begin{align*}
|\partial_x^\alpha \partial_{\xi_1}^{\beta_1} \partial_{\xi_2}^{\beta_2} \sigma(x, \xi_1, \xi_2)|
\le
C_{\alpha, \beta_1, \beta_2}(1+|\xi_1| +|\xi_2|)^m
\end{align*}
for all multi-indices $\alpha, \beta_1, \beta_2 \in \N^n_0 = \{0, 1, 2, \dots \}^n$. 
The other is a symbol class $BS^{(m_1, m_2)}_{0, 0}$, $m_1, m_2 \in \R$, 
consisting of all $\sigma(x, \xi_1, \xi_2) \in C^\infty((\R^n)^3)$ such that
\begin{align*}
|\partial_x^\alpha \partial_{\xi_1}^{\beta_1} \partial_{\xi_2}^{\beta_2} \sigma(x, \xi_1, \xi_2)|
\le
C_{\alpha, \beta_1, \beta_2}(1+|\xi_1|)^{m_1} (1+|\xi_2|)^{m_2}.
\end{align*}
For a symbol $\sigma$, the bilinear pseudo-differential operator $T_\sigma$ is defined by
\begin{align*}
T_\sigma (f_1, f_2)(x) 
=
\frac{1}{(2\pi)^{2n}} 
\int_{(\R^n)^2} 
e^{i x\cdot (\xi_1 +\xi_2)} \sigma(x, \xi_1, \xi_2) \widehat{f_1}(\xi_1) \widehat{f_2}(\xi_2)\ 
d\xi_1d\xi_2
\end{align*}
for $f_1, f_2 \in \Sh(\R^n)$.
%=============================================================================================

%=============================================================================================
The celebrated Calder\'on-Vaillancourt theorem \cite{CV} states that 
the linear pseudo-differential operators  with symbols $\sigma(x, \xi) \in C^\infty((\R^n)^2)$ satisfying
\begin{equation}\label{S000}
|\partial^\alpha_x \partial^\beta_\xi \sigma(x, \xi)|
\le
C_{\alpha, \beta}
\end{equation}
are bounded on $L^2$.
In contrast to this fact, bilinear pseudo-differential operators with symbols in $BS^0_{0, 0}$ are not always 
bounded  from $L^2 \times L^2$ to $L^1$. This interesting fact was first pointed out by B\'enyi-Torres \cite{BT-2}.
Then, Miyachi-Tomita \cite{MT-IUMJ} proved that 
all bilinear pseudo-differential operators with symbols in $BS^{-n/2}_{0, 0}$ are bounded 
from $L^2 \times L^2$ to $L^1$. By duality and interpolation, this implies that 
those operators are bounded from $L^{p_1} \times L^{p_2}$ to $L^p$ 
for $1 \le p \le 2 \le p_1, p_2 \le \infty$ with $1/p =1/p_1+1/p_2$. 
They also proved that 
the number $m=-n/2$ is the critical order in the sense that these boundedness do not hold if $m > -n/2$.
For the preceding results in the subcritical case $m< -n/2$, 
see B\'enyi-Bernicot-Maldonado-Naibo-Torres \cite{BBMNT} and Michalowski-Rule-Staubach \cite{MRS}.
%=============================================================================================

%=============================================================================================
Recently, Kato-Miyachi-Tomita \cite{KMT, KMT-2} proved that all bilinear pseudo-differential operators 
with symbols in $BS^{(m_1, m_2)}_{0, 0}$, $m_1, m_2 < 0$, $m_1+m_2 = -n/2$, 
are bounded from $L^{p_1} \times L^{p_2}$ to $L^p$ for $1 \le p \le 2 \le p_1, p_2 \le \infty$ 
and $1/p \le 1/p_1+1/p_2$ by using $L^2$-based amalgam spaces. 
Since
$BS^{-n/2}_{0, 0} \subset BS^{(m_1, m_2)}_{0, 0}$ for $\ m_1, m_2 \le 0\ \text{and}\ m_1+m_2 = -n/2$, 
this improves the result of \cite{MT} in terms of  symbol classes. 
They also pointed out that the condition $1/p =1/p_1+ 1/p_2$ is not always necessary for $S_{0, 0}$-type symbols. 
Quite recently, Hamada-Shida-Tomita \cite{HST} proved that 
all $T_\sigma$ with $\sigma \in BS^{(m_1, m_2)}_{0, 0}$, $m_1, m_2 < 0$, $m_1+m_2 = -n/2$, 
are bounded from $L^2 \times L^2$ to the Besov spaces $B^0_{p, q}$, $1\le p \le 2$, $q=1$, 
and showed the optimality of the ranges of  $p$ and $q$.
%=============================================================================================

%============================================================================================= 
 Next, we recall the Kato-Ponce inequality: 
\begin{equation*} 
\|(I-\Delta)^{s/2}(f_1 f_2)\|_{L^p}
\lesssim 
\|(I-\Delta)^{s/2} f_1\|_{L^{p_1}}
\|f_2\|_{L^{p_2}}
+
\|f_1\|_{L^{\widetilde{p}_1}}
\|(I-\Delta)^{s/2} f_2\|_{L^{\widetilde{p}_2}},
\end{equation*} 
where $1 < p_1, p_2, \widetilde{p}_1, \widetilde{p}_2 \le \infty$ and $1 < p < \infty$ satisfy 
$1/p = 1/p_1 + 1/p_2 = 1/\widetilde{p}_1 +1/\widetilde{p}_2$ and $s > 0$ 
(see, e.g., Kato-Ponce \cite{KT}, Muscalu-Schlag \cite{MS} and Grafakos-Oh \cite{GO}). 
The following Kato-Ponce type inequality also have been studied in many researches 
(see, e.g., B\'enyi-Torres \cite{BT-1}, Koezuka-Tomita \cite{KT} and Naibo-Thomson \cite{NT}):
\begin{equation*}
\|(I-\Delta)^{s/2}T_\sigma(f_1, f_2)\|_{X}
\lesssim 
\|(I-\Delta)^{s/2} f_1\|_{X_1}
\|f_2\|_{X_2}
+
\|f_1\|_{\widetilde{X}_1}
\|(I-\Delta)^{s/2} f_2\|_{\widetilde{X}_2},
\end{equation*}
where $X_i, \widetilde{X}_i$, $i= 1, 2$, are function spaces on $\R^n$. 
Since $T_\sigma (f_1, f_2) = f_1 f_2$ if $\sigma \equiv 1$, 
this can be regarded as a generalization of the Kato-Ponce inequality.
Under the restriction of exponents, 
the result of \cite{NT} reads as follows.

\begin{thmA}[{\cite[Theorem 1.1]{NT}}]
Let $1 \le p \le 2 \le p_1, p_2, \widetilde{p}_1, \widetilde{p}_2 \le \infty$ 
be such that 
$1/p = 1/p_1 + 1/p_2= 1/\widetilde{p}_1 + 1/\widetilde{p}_2$, 
$0< q \le \infty$ and $\sigma \in BS^{-n/2}_{0, 0}$. If $s> 0$, then 
\begin{align} \label{NT-est}
\|T_\sigma(f_1, f_2)\|_{B^s_{p, q}}
 \lesssim
 \|f_1\|_{B^{s}_{p_1, q}}
 \|f_2\|_{L^{p_2}}
 +
  \|f_1\|_{L^{\widetilde{p}_1}}
 \|f_2\|_{B^{s}_{\widetilde{p}_2, q}}
\end{align}
for all $f_1, f_2 \in \Sh(\R^n)$.
\end{thmA}
%=============================================================================================

The purpose of this paper is to prove the Kato-Ponce type inequality for bilinear pseudo-differential operators
 with $S_{0, 0}$-type symbols in the setting of Besov spaces. The main theorem is the following.

%=============================================================================================
\begin{thm} \label{main1}
Let 
$1\le p \le 2 \le p_1, p_2, \widetilde{p}_1, \widetilde{p}_2 \le \infty$, 
$0 < q, q_1, q_2, \widetilde{q}_1, \widetilde{q}_2 \le \infty$ 
and $s, s_1, s_2, \widetilde{s}_1, \widetilde{s}_2 \in \R$ 
be such that
\begin{equation} \label{pqs-conditions}
\frac{1}{p} \le \frac{1}{p_1} + \frac{1}{p_2},\  \frac{1}{\widetilde{p}_1} + \frac{1}{\widetilde{p}_2},
\quad
\frac{1}{q} =\frac{1}{q_1}+\frac{1}{q_2} =\frac{1}{\widetilde{q}_1}+\frac{1}{\widetilde{q}_2},
\quad
s=s_1 +s_2= \widetilde{s}_1 +\widetilde{s}_2.
\end{equation}
\begin{enumerate}
\item
Suppose $\sigma \in BS^{-n/2}_{0, 0}$. If $\widetilde{s}_1$, $s_2 < n/2$ and $s > -n/2$, then it holds that
\begin{align} \label{main-est-1}
 \|T_\sigma(f_1, f_2)\|_{B^s_{p, q}}
 \lesssim
 \|f_1\|_{B^{s_1}_{p_1, q_1}}
 \|f_2\|_{B^{s_2}_{p_2, q_2}}
 +
 \|f_1\|_{B^{\widetilde{s}_1}_{\widetilde{p}_1, \widetilde{q}_1}}
 \|f_2\|_{B^{\widetilde{s}_2}_{\widetilde{p}_2, \widetilde{q}_2}}
\end{align}
for all $f_1, f_2 \in \Sh(\R^n)$.
\item
Suppose that  
$m_1, m_2 \in \R$, 
$m_1 +m_2 = -n/2$ and $\sigma \in BS^{(m_1, m_2)}_{0, 0}$. 
If $\widetilde{s_1}< m_1+n/2$, $s_2 < m_2 + n/2$ and $s > -n/2$, 
then \eqref{main-est-1} holds for all $f_1, f_2 \in \Sh(\R^n)$.
\end{enumerate}
\end{thm}
%=============================================================================================

The following states the sharpness of the orders $\widetilde{s}_1, s_2$ and $s$.

%=============================================================================================
\begin{thm} \label{main2}
Let $p, p_1, p_2, \widetilde{p}_1, 
\widetilde{p}_2$, $q, q_1, q_2, \widetilde{q}_1, \widetilde{q}_2$ and $s, s_1, s_2, \widetilde{s}_1, \widetilde{s}_2$ 
satisfy \eqref{pqs-conditions}.
\begin{enumerate}
\item
 Suppose $\sigma \in BS^{-n/2}_{0, 0}$. If \eqref{main-est-1} holds, then 
 $\min \{s_i, \widetilde{s}_i\} \le n/2$ for $i=1, 2$ and $s \ge -n/2$.
\item
Suppose that 
$m_1, m_2 \in \R$, 
$m_1 +m_2 = -n/2$ and $\sigma \in BS^{(m_1, m_2)}_{0, 0}$. If \eqref{main-est-1} holds, then 
$\min \{ s_i, \widetilde{s}_i\} \le m_i + n/2$ for $i=1, 2$
 and $s \ge -n/2$.
\end{enumerate}
\end{thm}
%=============================================================================================

%=============================================================================================
We shall give some comments on the relations 
between Theorems \ref{main1}, \ref{main2} and preceding results. 
%=============================================================================================
First, there is a certain difference between the linear 
and bilinear cases. It is known that all linear pseudo-differential 
operators with symbols satisfying \eqref{S000} are bounded on 
the Besov space $B^s_{2, q}$ for all $ 0 < q \le \infty$ and $s \in \R$, 
where it should be emphasized that there is no restriction of $s$ 
(see Sugimoto \cite{Sugimoto}). 
On the other hand,
it follows from Theorem \ref{main1} (1) 
with $s_i = \widetilde{s}_i$, $i=1, 2$, 
that
if  $s_1, s_2 <n/2$, $s = s_1 +s_2 > -n/2$
and $\sigma \in BS^{-n/2}_{0, 0}$, then
\begin{equation} \label{Besovboundedness}
\|T_\sigma(f_1, f_2)\|_{B^s_{p, q}} 
\lesssim 
\|f_1\|_{B^{s_1}_{p_1, q_1}} \|f_2\|_{B^{s_2}_{p_2, q_2}}.
\end{equation}
In other words, all bilinear pseudo-differential operators 
with symbols in $BS^{-n/2}_{0, 0}$ are bounded from 
$B^{s_1}_{p_1, q_1} \times B^{s_2}_{p_2, q_2}$ to $B^{s}_{p, q}$ 
if $s_1,\ s_2< n/2$ and $s=s_1+s_2 > -n/2$. 
Theorem \ref{main2} implies that 
the boundedness does not hold 
if $s_1 > n/2$, $s_2 > n/2$ or $s_1 +s_2 < -n/2$.
Thus, we cannot remove the restrictions of $s_1, s_2$ and $s$ to have the boundedness 
in the bilinear case, and it is necessary to consider 
the Kato-Ponce type inequality without such restrictions.
%=============================================================================================
Secondly, Theorem \ref{main1} improves some previous results. 
It follows from \eqref{Besovboundedness} that
all $T_\sigma$ with $\sigma \in BS^{-n/2}_{0, 0}$
are bounded from
$B^{0}_{p_1, p_1} \times B^{0}_{p_2, p_2}$ to $B^{0}_{p, p}$, 
$1 \le p \le 2 \le p_1, p_2 \le \infty$, $1/p \le 1/p_1 + 1/p_2$, 
and consequently Theorem \ref{main1} (1) improves the $L^{p_1} \times L^{p_2} \to L^p$ boundedness 
since $B^0_{p, p} \hookrightarrow L^p$ and 
$L^{p_i} \hookrightarrow B^0_{p_i, p_i}$, $i=1, 2$. 
By Theorem \ref{main1} (2), 
we also have the same boundedness for 
all $T_\sigma$ with $\sigma \in BS^{(m_1, m_2)}_{0, 0}$, $m_1, m_2 < 0$, $m_1 + m_2 = -n/2$.
In particular, 
since $B^0_{2, 2} = L^2$, Theorem \ref{main1} includes the result of \cite{HST}.
Moreover, if $ s> - n/2$ and $\sigma \in BS^{-n/2}_{0, 0}$, then we have 
\begin{equation*}
\|T_\sigma(f_1, f_2)\|_{B^s_{p, q}}
 \lesssim
\|f_1\|_{B^{s}_{p_1, q}}
 \|f_2\|_{B^0_{p_2, \infty}}
 +
 \|f_1\|_{B^0_{p_1, \infty}}
 \|f_2\|_{B^{s}_{p_2, q}}
\end{equation*} 
and Theorem \ref{main1} also improves Theorem A 
since $L^p \hookrightarrow B^0_{p, \infty}$.
It should be mentioned that this inequality holds not only $s > 0$ but also $-n/2 < s  \le 0$.
%=============================================================================================
Finally, in contrast to the Kato-Ponce inequality, 
we can redistribute the fractional derivatives.
More precisely, 
for $s, s_i, \widetilde{s}_i \ge 0,\ i=1, 2$, 
with $s=s_1 +s_2 =\widetilde{s}_1 + \widetilde{s}_2$,
if 
\begin{multline} \label{Kato-Ponce-modoki}
\|(I-\Delta)^{s/2}(f_1 f_2)\|_{L^p}
\lesssim 
\|(I-\Delta)^{s_1/2} f_1\|_{L^{p_1}}
\|(I-\Delta)^{s_2/2}f_2\|_{L^{p_2}}
\\
+
\|(I-\Delta)^{\widetilde{s}_1/2}f_1\|_{L^{\widetilde{p}_1}}
\|(I-\Delta)^{\widetilde{s}_2/2} f_2\|_{L^{\widetilde{p}_2}},
\end{multline}
then 
$\min \{s_i, \widetilde{s}_i\} = 0$ 
for $i = 1, 2$ (see Remark \ref{rem-Sobolev}). 
Hence, if $s_1 \ge \widetilde{s}_1$ (automatically $s_2 \le \widetilde{s}_2$), 
then $s_1= \widetilde{s}_2 = s$ and $s_2 =\widetilde{s}_1 = 0$. 
On the other hand, in Theorem \ref{main1}, we have choices 
to redistribute the exponents $s_1, s_2, \widetilde{s}_1$ and $\widetilde{s}_2$.
%=============================================================================================

%=============================================================================================
The contents of this paper are as follows.
In Section \ref{section2},
we recall some preliminary facts and basic notations.
In Section \ref{section3},
we give some basic estimates and lemmas used in the proof of Theorem \ref{main1}.
In Sections \ref{section4} and \ref{section5},
we prove Theorem \ref{main1} and \ref{main2}, respectively.
%=============================================================================================

%=============================================================================================
%=============================================================================================

%=============================================================================================
%=============================================================================================
\section{Preliminaries}\label{section2}
For two nonnegative quantities $A$ and $B$,
the notation $A \lesssim B$ means that
$A \le CB$ for some unspecified constant $C>0$,
and $A \approx B$ means that
$A \lesssim B$ and $B \lesssim A$.
We denote by $|\Lambda|$ 
the number of  elements of a finite subset $\Lambda$ of $\Z^n$. 
For $1 \le p \le \infty$,
$p'$ is the conjugate exponent of $p$,
that is, $1/p+1/p'=1$.
We denote the usual inner product of $f, g \in L^2(\R^n)$ 
by $\la f, g \ra$.
%=============================================================================================

%=============================================================================================
Let $\Sh(\R^n)$ and $\Sh'(\R^n)$ be the Schwartz space of
rapidly decreasing smooth functions on $\R^n$ and its dual,
the space of tempered distributions, respectively.
We define the Fourier transform $\F f$
and the inverse Fourier transform $\F^{-1}f$
of $f \in \Sh(\R^n)$ by
\[
\F f(\xi)
=\widehat{f}(\xi)
=\int_{\R^n}e^{-ix\cdot\xi} f(x)\, dx
\quad \text{and} \quad
\F^{-1}f(x)
=\frac{1}{(2\pi)^n}
\int_{\R^n}e^{ix\cdot \xi} f(\xi)\, d\xi.
\]
For $m \in L^{\infty}(\R^n)$,
the Fourier multiplier operator $m(D)$ is defined by
$m(D)f=\F^{-1}[m\widehat{f}]$ for $f \in \Sh(\R^n)$.
%=============================================================================================

%=============================================================================================
We recall the definition of Besov spaces $B^s_{p, q}(\R^n)$.
Let $\psi_0 \in \Sh(\R^n) $ be such that 
$\supp \psi_0 \subset \{\xi \in \R^n : |\xi| \le 2\}$ and $\psi_0(\xi) = 1$ on $\{\xi \in \R^n  :  |\xi| \le 1\}$.
Set $\psi(\xi) = \psi_0(\xi) - \psi_0(2\xi)$ and $\psi_\ell(\xi) = \psi(2^{-\ell} \xi)$, $\ell \ge 1$.
Note that $\supp \psi_\ell \subset \{\xi \in \R^n : 2^{\ell-1} \le |\xi| \le 2^{\ell+1}\}$, $\ell \ge 1$, 
and $\sum_{\ell \ge 0} \psi_\ell(\xi) = 1$ for all $\xi \in \R^n$. 
For $0< p, q \le \infty$ and $s \in \R$, the Besov space $B^s_{p, q}(\R^n)$ 
consists of all $f \in \Sh^\prime(\R^n)$ such that
\begin{align*}
\|f\|_{B^s_{p, q}}=
\bigg(  \sum_{\ell=0}^\infty 2^{\ell s q} \|\psi_\ell(D)f\|^q_{L^p} \bigg)^{1/q} < \infty
\end{align*}
with usual modification when $q=\infty$.
It is well known that the definition of Besov spaces
$B_{p,q}^s(\R^n)$ does not depend on the choice of $\{\psi_\ell\}_{\ell \ge 0}$.
See \cite{Triebel} for more details on Besov spaces.

For $1 \le p \le \infty$ and $s \in \R$, 
the $L^p$-based Sobolev space $L^p_s(\R^n)$ consists of all $f \in \Sh^\prime(\R^n)$ such that
\begin{equation*}
 \|f\|_{L^p_s} = \|(I -\Delta)^{s/2} f \|_{L^p} < \infty,
\end{equation*}
where $(I- \Delta)^{s/2}f = \F^{-1}[ (1+|\cdot|^2)^{s/2} \widehat{f}]$.
For 
$1 \le p \le \infty$, $ 0< q \le \infty$, $s \in \R$ and $\epsilon > 0$, 
the embedding 
$L^p_{s+\epsilon}(\R^n) \hookrightarrow B^s_{p, q}(\R^n) \hookrightarrow L^p_{s-\epsilon}(\R^n)$
hold. For details on these embeddings, see \cite[Proposition 2.3.2/2, Theorem 2.3.8 and Proposition  2.5.7]{Triebel}.
%=============================================================================================
%=============================================================================================

%=============================================================================================
%=============================================================================================
\section{Key lemmas}\label{section3}
In this section, we shall give key lemmas which will be used later. 
We define
\begin{align*}
  S_R(f)(x) = R^n\int_{\R^n} \frac{|f(y)|}{(1+R|x-y|)^{n+1}}\, dy, \quad R > 0,
\end{align*}
and  write $S_1(f)(x)$ as  $S(f)(x)$.
The following lemma with $p =\widetilde{p}$ can be found in Kato \cite[Lemma 4.2]{Kato}.
%=============================================================================================

%=============================================================================================
\begin{lem} \label{lem-dyadic-est}
 Let $2 \le p \le \widetilde{p} \le \infty$ and $\varphi \in \Sh(\R^n)$.  Then, we have
 \begin{align*}
  \left\|
  \bigg(\sum_{\nu \in \Z^n} |\varphi(R^{-1}(D -\nu))f|^2 \bigg)^{1/2} 
  \right\|_{L^{\widetilde{p}}}
  \lesssim R^{n(1/2 + 1/p -1/\widetilde{p})} \|f\|_{L^p}
 \end{align*}
 for all $R \ge 1$. 
\end{lem}
%=============================================================================================

%=============================================================================================
\begin{proof}
Since the pointwise estimate
 \begin{equation} \label{square-est}
 \bigg(\sum_{\nu \in \Z^n} |\varphi(R^{-1}(D -\nu))f(x)|^2 \bigg)^{1/2} 
 \lesssim
 R^{n/2}S_R(|f|^2)(x)^{1/2}
 \end{equation}
 holds for $R \ge 1$ (see \cite[Lemma 3.1]{HST}), we have
 \begin{equation*}
 \left\|
  \bigg(\sum_{\nu \in \Z^n} |\varphi(R^{-1}(D -\nu))f|^2 \bigg)^{1/2} 
  \right\|_{L^{\widetilde{p}}}
  \lesssim
  R^{n/2} \|S_R(|f|^2)^{1/2}\|_{L^{\widetilde{p}}}
  =
  R^{n/2} \|S_R(|f|^2)\|_{L^{\widetilde{p}/2}}^{1/2}.
 \end{equation*}
Set $1/r =  2/\widetilde{p} - 2/p + 1$.  Our assumption implies that $1 \le p/2 \le \widetilde{p}/2 \le \infty$ and
$1 \le r \le \infty$. Therefore, it follows from Young's inequality
\begin{align*}
\|S_R(|f|^2)\|_{L^{\widetilde{p}/2}}^{1/2}
=
\| \zeta_R * |f|^2\|_{L^{\widetilde{p}/2}}^{1/2}
\le
\|\zeta_R\|_{L^r}^{1/2} \left\| |f|^2 \right\|_{L^{p/2}}^{1/2}
\approx
R^{n( 1/p - 1/\widetilde{p})} \|f\|_{L^p},
\end{align*}
where $\zeta_R(x) = R^n (1+R|x|)^{-n-1}$. This completes the proof.
\end{proof}
%=============================================================================================

%=============================================================================================
Let $\sigma \in BS^m_{0, 0} \cup BS^{(m_1, m_2)}_{0, 0}$.
We shall give the decomposition of $\sigma$. 
We use  
$\{\psi_j\}_{j \ge 0}$, which is the same as in the definition of Besov spaces, and  also use
$\varphi \in \Sh(\R^n)$ satisfying 
$\supp \varphi \subset [-1, 1]^n$
and
$\sum_{\nu \in \Z^n} \varphi(\xi- \nu) = 1$, 
$\xi \in \R^n$.
Then, we decompose $\sigma$ as 
\begin{align*} 
\sigma(x, \xi_1, \xi_2) 
=
\sum_{j \in \N_0} \sum_{\substack{\K \in (\N_0)^2 \\ \K= (k_1, k_2)}}  
\sum_{\substack{\Nu \in (\Z^n)^2 \\ \Nu= (\nu_1, \nu_2)}} 
\sigma_{j, \K, \Nu} (x, \xi_1, \xi_2),
\end{align*}
where
\begin{align*}
\sigma_{j, \K, \Nu} (x, \xi_1, \xi_2) 
&=  [\psi_j(D_x)\sigma](x, \xi_1, \xi_2)\psi_{k_1}(\xi_1)\psi_{k_2}(\xi_2)\varphi(\xi_1-\nu_1) \varphi(\xi_2-\nu_2)
\end{align*}
and
\begin{align*}
[\psi_j(D_x)\sigma](x, \xi_1, \xi_2)
=
\int_{\R^n} \F^{-1}\psi_j(y) \sigma(x-y, \xi_1, \xi_2)\, dy.
\end{align*}
We remark that the Fourier transform of $T_{\sigma_{j, \K, \Nu}}(f_1, f_2)$ is given by
\begin{align} \label{FT-T}
\begin{split}
 &\F[T_{\sigma_{j, \K, \Nu}}(f_1, f_2)](\zeta)
 \\
 &= \frac{1}{(2\pi)^{2n}}\int_{(\R^n)^2}
   \psi_{j}(\zeta- (\xi_1+\xi_2)) [\F_x\sigma](\zeta-(\xi_1+\xi_2), \xi_1, \xi_2)\\
 &\qquad\qquad\qquad \times  \psi_{k_1}(\xi_1)\psi_{k_2}(\xi_2)
   \varphi(\xi_1-\nu_1)\varphi(\xi_2-\nu_2)
   \widehat{f_1}(\xi_1) \widehat{f_2}(\xi_2)\,d\xi_1d\xi_2,
\end{split}
\end{align}
where $\F_x\sigma$ denotes the partial Fourier transform
of $\sigma(x,\xi_1,\xi_2)$ with respect to  the $x$-variable.
If $|\zeta -\xi_1-\xi_2| \le 2^{j+1}$ and $\xi_i \in \nu_i + [-1, 1]^n$, $i =1, 2$, 
then $\zeta \in \nu_1+\nu_2 + [-2^{j+2}, 2^{j+2}]^n$.
From this, we see that
\begin{align} \label{supp-FT}
\supp \F[T_{\sigma_{j, \K, \Nu}}(f_1, f_2)] \subset \nu_1+\nu_2 + [-2^{j+2}, 2^{j+2}]^n.
\end{align}
%=============================================================================================

%=============================================================================================
\begin{lem}[{\cite[Lemma 3.2]{HST}}] \label{lem-symb-est}
Let $m, m_1, m_2 \in \R$ and $N \ge 0$.
\begin{enumerate}
\item
 If $\sigma \in BS^{m}_{0, 0}$, then 
\begin{equation*}
|T_{\sigma_{j, \K,\Nu}}(f_1, f_2)(x) | \lesssim 2^{\max \{k_1, k_2 \}m -jN}  S(f_1)(x) S(f_2)(x)
\end{equation*}
 for all $j \in \N_0$, $\K = (k_1, k_2) \in (\N_0)^2$ and $\Nu \in (\Z^n)^2$.

\item  If $\sigma \in BS^{(m_1, m_2)}_{0, 0}$, then 
\begin{equation*}
|T_{\sigma_{j, \K,\Nu}}(f_1, f_2)(x) | \lesssim 2^{k_1m_1 +k_2m_2 -jN}  S(f_1)(x) S(f_2)(x)
\end{equation*}
for all $j \in \N_0$, $\K = (k_1, k_2) \in (\N_0)^2$ and $\Nu \in (\Z^n)^2$.
\end{enumerate}
\end{lem}
%=============================================================================================

%=============================================================================================
\begin{rem}
The assertion (1) is not given in \cite{HST}, 
but it follows from the same argument as in \cite[Lemma 3.2]{HST}. 
In fact, 
since 
$1+|\xi_1| +|\xi_2| \approx 2^{\max \{k_1, k_2\}}$ 
if $\xi_i \in \supp \psi_{k_i}$, $i= 1, 2$, and $|\F^{-1}\psi_j(x)| \lesssim 2^{jn} (1+ 2^j|x|)^{-N-n-1}$, we obtain
\begin{equation*}
|\partial_{\xi_1}^{\beta_1}\partial_{\xi_2}^{\beta_2} \sigma_{j, \K, \Nu}(x, \xi_1, \xi_2)| 
\lesssim 
2^{\max \{k_1, k_2\} m- jN}
\end{equation*}
instead of \cite[Lemma 3.2 (1)]{HST}, which gives the desired assertion.
\end{rem}
%=============================================================================================

%=============================================================================================
The following lemma plays an important role in the proof of Theorem \ref{main1}.
The proof is essentially the  same as in \cite[Lemma 3.3]{HST}, and  
the basic idea goes back to \cite[Lemma 3.6]{MT}. 
%=============================================================================================

%=============================================================================================
\begin{lem} \label{lem-Lp1Lp2Lr-est}
Let $2 \le r, p_1, p_2 \le \infty$, $1-1/p_1-1/p_2 \le 1/r \le 1/2$, and $N \ge 0$.
\begin{enumerate}
\item
If $\sigma \in BS^{m}_{0, 0}$, then the estimate
 \begin{align*} 
  &\Big(
  \sum_{\nu_1 \in \Lambda} \sum_{\nu_2 \in \Z^n} + 
  \sum_{\nu_1 \in \Z^n} \sum_{\nu_2 \in \Lambda} +
  \sum_{\mu \in \Lambda} \sum_{\nu_1+\nu_2 = \mu}
  \Big)
  |\la T_{\sigma_{j, \K, \Nu}}(f_1, f_2), g \ra|
  \\
 &\qquad\qquad\qquad\qquad\quad
 \lesssim 2^{\max \{k_1, k_2 \}m -jN}
 |\Lambda|^{1/2} \|f_1\|_{L^{p_1}}\|f_2\|_{L^{p_2}} \|g\|_{L^r}
 \end{align*}
 holds for all $j \in \N_0$, $\K = (k_1, k_2) \in (\N_0)^2$ and all finite sets $\Lambda \subset \Z^n$.
\item
If $\sigma \in BS^{(m_1, m_2)}_{0, 0}$, then the estimate
 \begin{align*} 
  &\Big(
  \sum_{\nu_1 \in \Lambda} \sum_{\nu_2 \in \Z^n} + 
  \sum_{\nu_1 \in \Z^n} \sum_{\nu_2 \in \Lambda} +
  \sum_{\mu \in \Lambda} \sum_{\nu_1+\nu_2 = \mu}
  \Big)
  |\la T_{\sigma_{j, \K, \Nu}}(f_1, f_2), g \ra|
 \\
 &\qquad\qquad\qquad\qquad\quad
 \lesssim 2^{k_1 m_1+k_2 m_2 -jN}
 |\Lambda|^{1/2} \|f_1\|_{L^{p_1}}\|f_2\|_{L^{p_2}} \|g\|_{L^r}
 \end{align*}
 holds for all $j \in \N_0$, $\K = (k_1, k_2) \in (\N_0)^2$ and all finite sets $\Lambda \subset \Z^n$.
 \end{enumerate}
\end{lem}
%=============================================================================================

%=============================================================================================
\begin{proof}
 Let $\widetilde{\varphi} \in \Sh(\R^n)$ be such that 
  $\supp \widetilde{\varphi} \subset [-2, 2]^n$ and $\widetilde{\varphi}(\xi) = 1$ on $ [-1, 1]^n$.  
 Set
\begin{equation*}
  f_{i, \nu_i} = \widetilde{\varphi}(D-\nu_i)f_i, 
  \quad
  g_{j, \mu} = \widetilde{\varphi}(2^{-(j+2)}(D-\mu))g
\end{equation*}
 for $\nu_i \in \Z^n$, $i=1,2$, $\mu \in \Z^n$ and $j \in \N_0$.
Since $\varphi \widetilde{\varphi} = \varphi$, we have 
$T_{\sigma_{j, \K, \Nu}}(f_1, f_2) = T_{\sigma_{j, \K, \Nu}}(f_{1, \nu_1}, f_{2, \nu_2})$ 
for $\Nu = (\nu_1, \nu_2) \in (\Z^n)^2$. Hence, by  \eqref{supp-FT}, 
 \begin{align} \label{rep-trilin-T}
  \la T_{\sigma_{j, \K, \Nu}}(f_1, f_2), g\ra 
  = 
  \la T_{\sigma_{j, \K, \Nu}}( f_{1, \nu_1}, f_{2, \nu_2}), g_{j, \nu_1 + \nu_2}\ra.
\end{align}
It follows from \eqref{rep-trilin-T} and 
Lemma \ref{lem-symb-est} (1) with $N$ replaced by $\widetilde{N} = N + n(1/r +1/p_1 +1/p_2 -1/2)$ that
\begin{align} \label{sum-T-est}
\begin{split}
 &\sum_{\nu_1} \sum_{\nu_2} |\la T_{\sigma_{j, \K, \Nu}}(f_1, f_2), g \ra|
 =
 \sum_{\nu_1} \sum_{\nu_2} |\la T_{\sigma_{j, \K, \Nu}}(f_{1, \nu_1}, f_{2, \nu_2}), g_{j, \nu_1+ \nu_2} \ra|
 \\
 &\lesssim
 2^{\max \{k_1, k_2\}m -j\widetilde{N}}  
 \int_{\R^n} \sum_{\nu_1} \sum_{\nu_2} S(f_{1, \nu_1})(x)S(f_{2, \nu_2})(x) |g_{j, \nu_1 +\nu_2}(x)| \ dx
 \end{split}
\end{align}
if $\sigma \in BS^m_{0, 0}$, and
 \eqref{sum-T-est} with $\max \{k_1, k_2\}m$ replaced by $k_1m_1+k_2m_2$ 
 follows from \eqref{rep-trilin-T} and Lemma \ref{lem-symb-est} (2) if $\sigma \in BS^{(m_1, m_2)}_{0, 0}$. 
%=============================================================================================

%=============================================================================================
\medskip
\noindent
{\it Estimates for the sums $\sum_{\nu_1 \in \Lambda} \sum_{\nu_2 \in \Z^n} 
and \sum_{\nu_1 \in \Z^n} \sum_{\nu_2 \in \Lambda}$}.
By symmetry, we only consider the estimate for the former sum.
By the Cauchy-Schwarz inequality, 
\begin{align*}
&\sum_{\nu_1 \in \Lambda} \sum_{\nu_2 \in \Z^n} S(f_{1, \nu_1})(x)S(f_{2, \nu_2})(x) |g_{j, \nu_1 +\nu_2}(x)|
\\
&\le
\sum_{\nu_1 \in \Lambda} S(f_{1, \nu_1})(x) 
\Big(\sum_{\nu_2 \in \Z^n} S(f_{2, \nu_2})(x)^2 \Big)^{1/2} 
\Big(\sum_{\nu_2 \in \Z^n}| g_{j, \nu_1 +\nu_2}(x)|^2 \Big)^{1/2} 
\\
&\le
|\Lambda|^{1/2}
\Big(\sum_{\nu_1 \in \Z^n} S(f_{1, \nu_1})(x)^2 \Big)^{1/2} 
\Big(\sum_{\nu_2 \in \Z^n} S(f_{2, \nu_2})(x)^2 \Big)^{1/2} 
\Big(\sum_{\mu \in \Z^n} |g_{j, \mu}(x)|^2 \Big)^{1/2}.
\end{align*}
By \eqref{square-est}, we have
\begin{align*}
\Big(\sum_{\nu_i \in \Z^n} S(f_{i, \nu_i})(x)^2 \Big)^{1/2} 
&\lesssim
\Big(\sum_{\nu_i \in \Z^n} S(|f_{i, \nu_i}|^2)(x) \Big)^{1/2} 
=
\bigg( S\Big( \sum_{\nu_i \in \Z^n}|f_{i, \nu_i}|^2\Big)(x) \bigg)^{1/2} 
\\
&\lesssim
S(S(|f_i|^2))(x)^{1/2} 
\approx
S(|f_i|^2)(x)^{1/2}
\end{align*}
for $i =1, 2$.
Thus, 
\begin{align*}
&\int_{\R^n} 
\sum_{\nu_1 \in \Lambda}\sum_{\nu_2 \in \Z^n} 
S(f_{1, \nu_1})(x)S(f_{2, \nu_2})(x) |g_{j, \nu_1 +\nu_2}(x)|\ 
dx
\\
&\lesssim
|\Lambda|^{1/2} 
\int_{\R^n} 
S(|f_1|^2)(x)^{1/2}S(|f_2|^2)(x)^{1/2} 
\Big(\sum_{\mu \in \Z^n} |g_{j, \mu}(x)|^2 \Big)^{1/2}
\ dx
\\
&\le
|\Lambda|^{1/2} 
\|S(|f_1|^2)^{1/2}\|_{L^{p_1}} \|S(|f_2|^2)^{1/2}\|_{L^{p_2}} 
\Big\|\Big(\sum_{\mu \in \Z^n} |g_{j, \mu}|^2 \Big)^{1/2} \Big\|_{L^{\widetilde{r}}},
\end{align*}
where we used H\"older's inequality with $1/p_1+1/p_2 +1/\widetilde{r} = 1$ in the last inequality. 
Since $2 \le p_i \le \infty$, $i=1, 2$, it follows from Young's inequality that
\begin{equation*}
\|S(|f_i|^2)^{1/2}\|_{L^{p_i}} = \|S(|f_i|^2)\|_{L^{p_i/2}}^{1/2} \lesssim \||f_i|^2\|_{L^{p_i/2}}^{1/2} = \|f_i\|_{L^{p_i}},
\quad
i =1, 2.
\end{equation*}
Moreover, our assumption implies that
$0 \le 1/ \widetilde{r} = 1 - 1/p_1 -1/p_2 \le 1/r \le 1/2$, 
namely, $2 \le r \le \widetilde{r} \le \infty$.
Hence, by Lemma \ref{lem-dyadic-est}, we have
\begin{equation*}
\Big\|\Big(\sum_{\mu \in \Z^n} |g_{j, \mu}|^2 \Big)^{1/2} \Big\|_{L^{\widetilde{r}}}
\lesssim
2^{jn(1/2+1/r-1/\widetilde{r})} \|g\|_{L^r} = 2^{j(\widetilde{N}- N)}\|g\|_{L^r}.
\end{equation*}
Combining these estimates, we get the desired estimate.
%=============================================================================================
%=============================================================================================

%=============================================================================================
%=============================================================================================
\medskip
\noindent
{\it Estimate for the sum $\sum_{\mu \in \Lambda} \sum_{\nu_1 +\nu_2 =\mu}$}. 
By the Cauchy-Schwarz inequality, we have
\begin{align*}
&\sum_{\mu \in \Lambda} 
\sum_{\nu_1+\nu_2 = \mu} 
S(f_{1, \nu_1})(x) S(f_{2, \nu_2})(x) |g_{j, \mu}(x)|
\\
&\le
\sum_{\mu \in \Lambda}  
\Big(\sum_{\nu_1 \in \Z^n} S(f_{1, \nu_1})(x)^2 \Big)^{1/2} 
\Big(\sum_{\nu_1 \in \Z^n} S(f_{2, \mu-\nu_1})(x)^2 \Big)^{1/2} 
|g_{j, \mu}(x)|
\\
&\le
|\Lambda|^{1/2}
\Big(\sum_{\nu_1 \in \Z^n} S(f_{1, \nu_1})(x)^2 \Big)^{1/2} 
\Big(\sum_{\nu_2 \in \Z^n} S(f_{2, \nu_2})(x)^2 \Big)^{1/2} 
\Big(\sum_{\mu \in \Z^n} |g_{j, \mu}(x)|^2 \Big)^{1/2}.
\end{align*}
The rest of the proof is the same as above. The proof  is complete.
\end{proof}
%=============================================================================================
%=============================================================================================

%=============================================================================================
%=============================================================================================
\section{Proof of Theorem \ref{main1}}\label{section4}
In this section, we shall prove Theorem \ref{main1}.
We assume that $p, q, s$ and other exponents are the same as in Theorem \ref{main1}, 
and that all $\psi_\ell$, $\ell \ge 0$, are real-valued. 

By duality, we see that
\begin{align*}
 \|T_{\sigma}(f_1, f_2)\|_{B^s_{p, q}}
&=
 \Big( \sum_{\ell \ge 0} 2^{\ell s q} \| \psi_\ell(D)T_{\sigma}(f_1, f_2) \|_{L^p}^q\Big)^{1/q}
 \\
&=
\bigg\{ 
\sum_{\ell \ge 0} 2^{\ell s q} 
\Big( 
\sup_{\|g\|_{L^{p^\prime}} = 1} | \la \psi_\ell(D)T_{\sigma}(f_1, f_2), g \ra |
\Big)^q 
\bigg \}^{1/q}
\\
&=
\bigg\{ 
\sum_{\ell \ge 0} 2^{\ell s q} 
\Big( \sup_{\|g\|_{L^{p^\prime}} = 1} | \la T_{\sigma}(f_1, f_2), g_\ell \ra |
\Big)^q 
\bigg\}^{1/q},
\end{align*}
where $g_\ell = \psi_\ell(D)g$. 
Next, we decompose $\sigma$ as follows:
\begin{align*}
 \sigma 
 = \sum_{j \ge 0} \sum_{\substack{\K =(k_1, k_2) \\ k_1, k_2 \ge 0}} \sum_{\Nu \in (\Z^n)^2} \sigma_{j, \K, \Nu}
 = \sum_{j \ge 0} \Big( \sum_{k_1 \ge k_2} + \sum_{k_1 < k_2} \Big) \sum_{\Nu \in (\Z^n)^2} \sigma_{j, \K, \Nu}.
\end{align*}
By symmetry, it suffices to consider the former sum. 
Thus, our goal is to show
\begin{multline} \label{goal}
 \bigg\{ 
 \sum_{\ell \ge 0} 2^{\ell s q} 
 \Big( \sup_{\|g\|_{L^{p^\prime}} = 1} 
 \Big| 
 \sum_{j \ge 0} \sum_{k_1 \ge k_2}\sum_{\Nu \in (\Z^n)^2}\la T_{\sigma_{j, \K, \Nu}}(f_1, f_2), g_\ell \ra 
 \Big|
 \Big)^q 
 \bigg\}^{1/q}
 \\
 \lesssim
  \|f_1\|_{B^{s_1}_{p_1, q_1}}
 \|f_2\|_{B^{s_2}_{p_2, q_2}}.
\end{multline}

Let $\widetilde{\psi}_0, \widetilde{\psi}  \in \Sh(\R^n)$ be such that 
$\supp \widetilde{\psi}_0 \subset \{ |\xi| \le 4 \}$, 
$\widetilde{\psi}_0(\xi) = 1$ on $\{|\xi| \le 2\}$, 
$\supp \widetilde{\psi} \subset \{1/4 \le |\xi| \le 4\}$ 
and 
$\widetilde{\psi}(\xi) = 1$ on $\{1/2 \le |\xi| \le 2\}$,
and set $\widetilde{\psi}_k(\xi) = \widetilde{\psi}(2^{-k}\xi)$ for $k \ge 1$. 
Then, the estimate
$\Big( \sum_{k \ge 0} 2^{k s q} \| \widetilde{\psi}_k(D)f \|_{L^p}^q\Big)^{1/q} \lesssim \|f\|_{B^s_{p, q}}$ 
holds
for $0 < p, q \le \infty$, $s \in \R$ and $f \in \Sh(\R^n)$
(see \cite[Proof of Proposition 2.3.2/1]{Triebel}).
Set $f_{i, k_i} = \widetilde{\psi}_{k_i}(D)f_i$, $i=1, 2$, and
then, the identity $\psi_{k_i} \widetilde{\psi}_{k_i}= \psi_{k_i}$ implies that
$
\la T_{\sigma_{j, \K, \Nu}}(f_1, f_2), g_\ell \ra
=
\la T_{\sigma_{j, \K, \Nu}}(f_{1, k_1}, f_{2, k_2}), g_\ell \ra.
$
We divide the sum into the following three parts:
\begin{align*}
&\sum_{j \ge 0}
\sum_{k_1 \ge k_2}
\sum_{\Nu \in (\Z^n)^2}
\la T_{\sigma_{j, \K, \Nu}}(f_1, f_2), g_\ell \ra
=
\sum_{j \ge 0}
\sum_{k_1 \ge k_2}
\sum_{\Nu \in (\Z^n)^2}
\la T_{\sigma_{j, \K, \Nu}}(f_{1, k_1}, f_{2, k_2}), g_\ell \ra
\\
&=
\Big( 
\sum_{\substack{j \ge k_1-3 \\ k_1 \ge k_2}} 
+
\sum_{\substack{j < k_1-3 \\ k_2 \le k_1 \le k_2 + 3}}
+
\sum_{\substack{j < k_1-3 \\  k_1 > k_2 + 3}}
\Big)
\sum_{\Nu \in (\Z^n)^2}\la T_{\sigma_{j, \K, \Nu}}(f_{1, k_1}, f_{2, k_2}), g_\ell \ra
\\
&= A_{1, \ell} + A_{2, \ell} +A_{3, \ell}.
\end{align*}
%=============================================================================================

%=============================================================================================
\begin{proof}[Proof of the assertion (1)]
We first give the proof of the assertion (1). 
We assume that $s_2 < n/2$, $s = s_1 +s_2 > -n/2$ and $\sigma \in BS^{-n/2}_{0, 0}$.
%=============================================================================================

%=============================================================================================
\medskip
\noindent
{\it Estimate for $A_{1, \ell}$}. 
%We remark that the assumptions $s_2 < n/2$ and $s = s_1 +s_2 > -n/2$ is not necessary in this part.
By \eqref{FT-T}, we have
\begin{equation*}
\supp \F[T_{\sigma_{j, \K, \Nu}}(f_{1, k_1}, f_{2, k_2})] 
\subset 
\{|\zeta| \le 2^{j+6}\}, \quad j \ge k_1-3,\  k_1 \ge k_2.
\end{equation*}
Since $\supp \widehat{g_\ell} \subset \{2^{\ell-1} \le |\zeta| \le 2^{\ell+1}\}$, $\ell \ge 1$, we have 
$\la T_{\sigma_{j, \K, \Nu}}(f_{1, k_1}, f_{2, k_2}), g_\ell \ra = 0$ if $j \le \ell-7$. 
Moreover, if $\supp \varphi(\cdot -\nu_2) \cap \supp \psi_{k_2} = \emptyset$, 
then $\sigma_{j, \K, \Nu} = 0$ , and consequently $\la T_{\sigma_{j, \K, \Nu}}(f_{1, k_1}, f_{2, k_2}), g_\ell \ra = 0$.
From these observations, we have
\begin{equation*}
A_{1, \ell}
= 
\sum_{j \ge \ell-6}  \sum_{\substack{k_1 \le j+3 \\ k_2 \le k_1}} 
\sum_{\nu_1 \in \Z^n}\sum_{\nu_2 \in \Lambda_{k_2}}
\la T_{\sigma_{j, \K, \Nu}}(f_{1, k_1}, f_{2, k_2}), g_\ell \ra,
\end{equation*}
where 
\begin{align} \label{set-k2}
 \Lambda_{k_2} = \{\nu_2 \in \Z^n : \supp \varphi(\cdot -\nu_2) \cap \supp \psi_{k_2} \ne \emptyset\}.
\end{align}
Note that $|\Lambda_{k_2}| \lesssim 2^{k_2n}$. 
It follows from Lemma \ref{lem-Lp1Lp2Lr-est} (1) with $r = p^\prime$ and $N > \max\{ n/2, s+n/2\}$ that
\begin{align}\label{A1-est}
\begin{split}
 |A_{1, \ell}|
 &\lesssim
 \sum_{j \ge \ell-6}
 \sum_{\substack{k_1 \le j+3 \\ k_2 \le k_1}} 
 2^{-(k_1-k_2)n/2 -jN} \|f_{1, k_1}\|_{L^{p_1}} \|f_{2, k_2}\|_{L^{p_2}} \|g_\ell\|_{L^{p^\prime}}
 \\
 &\lesssim
 \bigg(
 \sum_{j \ge \ell-6}
 \sum_{\substack{k_1 \le j+3 \\ k_2 \le k_1}} 
 2^{-k_1(s_1 +n/2) +k_2(n/2 - s_2) -  jN}
 \bigg)
 \|f_1\|_{B^{s_1}_{p_1, q_1}} \|f_2\|_{B^{s_2}_{p_2, q_2}} \|g\|_{L^{p^\prime}},
 \end{split}
\end{align}
where we used the fact that
$\|f_{i, k_i}\|_{L^{p_i}} 
%\le
%2^{-k_is_i} \|f_i\|_{B^{s_i}_{p_i, \infty}}
\le
2^{-k_is_i} \|f_i\|_{B^{s_i}_{p_i, q_i}} $
for 
$i =1, 2$
in the second inequality.
Since $s_2 < n/2$ and $s=s_1 +s_2> -n/2$, 
we have 
\begin{align*}
&\sum_{j \ge \ell-6}
\sum_{k_1 \le j+3} 
\sum_{k_2 \le k_1}
2^{-k_1(s_1 +n/2) +k_2(n/2 - s_2) -  jN}
\approx
\sum_{j \ge \ell-6} 2^{-jN} \sum_{k_1 \le j+3} 2^{-k_1s}
\\
&\le
\sum_{j \ge \ell-6}2^{-jN} \sum_{k_1 \le j+3} 2^{k_1n/2}
\approx
\sum_{j \ge \ell-6} 
 2^{-j(N-n/2)}  
\lesssim 
2^{-\ell(N -n/2)}.
\end{align*}
Therefore, 
the left hand side of  \eqref{goal} concerning $A_{1, \ell}$ is estimated by
\begin{align*}
%\bigg\{ \sum_{\ell \ge 0} 2^{\ell sq} \Big(\sup_{\|g\|_{L^{p^\prime}} = 1} |A_{1, \ell}| \Big)^q \bigg\}^{1/q}
%&\lesssim
\Big( \sum_{\ell \ge 0}2^{-\ell (N-s-n/2) q} \Big)^{1/q} \|f_1\|_{B^{s_1}_{p_1, q_1}} \|f_2\|_{B^{s_2}_{p_2, q_2}}
\approx
\|f_1\|_{B^{s_1}_{p_1, q_1}} \|f_2\|_{B^{s_2}_{p_2, q_2}}.
\end{align*}
%=============================================================================================

%=============================================================================================
\medskip
\noindent
{\it Estimate for $A_{2, \ell}$}.
If $j < k_1-3$ and $k_1 \ge k_2$, then
\begin{equation*}
\supp \F[T_{\sigma_{j, \K, \Nu}}(f_{1, k_1}, f_{2, k_2})] \subset \{|\zeta| \le 2^{k_1 +3}\},
\end{equation*}
and this implies that $\la T_{\sigma_{j, \K, \Nu}}(f_{1, k_1}, f_{2, k_2}), g_\ell \ra = 0$ for $k_1 \le\ell-4$.
Moreover, it follows from \eqref{supp-FT} and the fact $\supp \widehat{g}_\ell \subset \supp \psi_\ell$ that
if $(\nu_1+ \nu_2+ [-2^{j+2}, 2^{j+2}]^n) \cap \supp \psi_\ell = \emptyset $, 
then $\la T_{\sigma_{j, \K, \Nu}}(f_{1, k_1}, f_{2, k_2}), g_\ell \ra = 0$.
Therefore, $A_{2, \ell}$ is represented by
\begin{align*}
A_{2, \ell} 
= 
\sum_{\substack{k_1 \ge \ell-3 \\ k_1-3 \le k_2 \le k_1}}
\sum_{j < k_1-3} \sum_{\mu \in \Lambda_{j, \ell}} 
\sum_{\nu_1+\nu_2 =\mu}
\la T_{\sigma_{j, \K, \Nu}}(f_{1, k_1}, f_{2, k_2}), g_\ell \ra,
\end{align*} 
where 
$\Lambda_{j, \ell} = \{\mu \in \Z^n : (\mu + [-2^{j+2}, 2^{j+2}]^n) \cap \supp \psi_\ell \ne \emptyset\}$.
Since $|\Lambda_{j, \ell}| \lesssim 2^{(j+\ell)n}$, 
it follows from Lemma \ref{lem-Lp1Lp2Lr-est} (1) with $r = p^\prime$ and $N > n/2$ that
\begin{align} \label{A2-est}
 \begin{split}
 |A_{2, \ell}|
 &\lesssim
 \sum_{\substack{k_1 \ge \ell-3 \\ k_1-3 \le k_2 \le k_1}}
 \sum_{j < k_1-3}  
 2^{-(k_1-\ell)n/2 -j(N-n/2)} \|f_{1, k_1}\|_{L^{p_1}} \|f_{2, k_2}\|_{L^{p_2}} \|g_\ell\|_{L^{p^\prime}}
 \\
 &\lesssim
 \sum_{\substack{k_1 \ge \ell-3 \\ k_1-3 \le k_2 \le k_1}} 
 2^{-(k_1-\ell)n/2} \|f_{1, k_1}\|_{L^{p_1}} \|f_{2, k_2}\|_{L^{p_2}} \|g\|_{L^{p^\prime}}.
 \end{split}
\end{align}
By a change of variables, we can write the last quantity in \eqref{A2-est} as 
\begin{equation*}
\sum_{0 \le k \le 3} 
\sum_{k_1 \ge \ell -3} 
2^{-(k_1-\ell)n/2} \|f_{1, k_1}\|_{L^{p_1}} \|f_{2, k_1 - k}\|_{L^{p_2}} \|g\|_{L^{p^\prime}}.
\end{equation*} 
Hence, the left hand side of  \eqref{goal} concerning $A_{2, \ell}$ is estimated by
\begin{align} \label{A2l-est}
\begin{split}
\sum_{0 \le k \le 3} 
2^{ks_2}
\bigg\{ 
\sum_{\ell \ge 0} 
\Big( 
\sum_{k_1 \ge \ell -3} 
2^{(\ell - k_1) (s +n/2)}  2^{k_1s_1}\|f_{1, k_1}\|_{L^{p_1}} 2^{(k_1-k)s_2} \|f_{2, k_1-k}\|_{L^{p_2}} 
\Big)^q 
\bigg\}^{1/q}.
\end{split}
\end{align}
Hereafter, we only consider the case $k=0$, but our argument works for the other cases.
If $q \le 1$, then \eqref{A2l-est} with $k=0$ is estimated by
\begin{align*}
 &\bigg( 
 \sum_{\ell \ge 0} 
 \sum_{k_1 \ge \ell -3} 
 2^{(\ell - k_1) (s +n/2) q}  2^{k_1s_1 q}\|f_{1, k_1}\|^q_{L^{p_1}} 2^{k_1 s_2 q} \|f_{2, k_1}\|^q_{L^{p_2}} 
 \bigg)^{1/q}
 \\
 &\approx
 \bigg( \sum_{k_1 \ge0} 2^{k_1s_1 q} \|f_{1, k_1}\|_{L^{p_1}}^q 2^{k_1 s_2 q}\|f_{2, k_1}\|_{L^{p_2}}^q \bigg)^{1/q}
 \\
 &\le
 \Big(
 \sum_{k_1 \ge 0}
 2^{k_1s_1 q_1} \|f_{1, k_1}\|_{L^{p_1}}^{q_1}
 \Big)^{1/q_1}
 \Big(
 \sum_{k_1 \ge 0}
 2^{k_1s_2 q_2} \|f_{2, k_1}\|_{L^{p_2}}^{q_2}
 \Big)^{1/q_2}
 \lesssim
 \|f_1\|_{B^{s_1}_{p_1, q_1}} \|f_2\|_{B^{s_2}_{p_2, q_2}},
\end{align*}
where we used the assumption $s= s_1+s_2 > -n/2$ and H\"older's inequality. 
On the other hand, 
if $q > 1$, 
then it follows from Young's inequality that \eqref{A2l-est} with $k=0$ is estimated by
\begin{align*}
 &\bigg\{ 
  \sum_{\ell \ge 0}  
  \Big( 
  \sum_{k_1 \ge 0} 
  2^{-|\ell-k_1|(s+n/2)} 
  2^{k_1s_1}\|f_{1, k_1}\|_{L^{p_1}}
  2^{k_1s_2}\|f_{2, k_1}\|_{L^{p_2}} 
  \Big)^{q} 
 \bigg\}^{1/q} 
 \\
 &\le
 \Big(\sum_{\ell \ge 0} 2^{-|\ell|(s+n/2)}\Big) 
 \Big( \sum_{k_1 \ge 0}   2^{k_1s_1q}\|f_{1, k_1}\|_{L^{p_1}}^q 2^{k_1s_2q}\|f_{2, k_1}\|_{L^{p_2}}^q  \Big)^{1/q}
 \\
 &\lesssim
 \|f_1\|_{B^{s_1}_{p_1, q_1}}
 \|f_2\|_{B^{s_2}_{p_2, q_2}},
\end{align*}
where we used the assumption $s> -n/2$ and H\"older's inequality.
%=============================================================================================

%=============================================================================================
\medskip
\noindent
{\it Estimate for $A_{3, \ell}$}. 
We see that if $j < k_1-3$ and $k_1 >k_2+3$, then
\begin{equation*}
 \supp \F[T_{\sigma_{j, \K, \Nu}}(f_{1, k_1}, f_{2, k_2})]
 \subset
 \{2^{k_1-2} \le |\zeta| \le 2^{k_1+2}\}.
\end{equation*}
This implies that $\la T_{\sigma_{j, \K, \Nu}}(f_{1, k_1}, f_{2, k_2}), g_\ell  \ra = 0$ if $|k_1 -\ell| \ge 3$.
Furthermore, as mentioned above, 
if $\supp \varphi(\cdot -\nu_2) \cap \supp \psi_{k_2} = \emptyset$, 
then $\la T_{\sigma_{j, \K, \Nu}}(f_{1, k_1}, f_{2, k_2}), g_\ell  \ra = 0$.
Hence, $A_{3, \ell}$ can be written as
\begin{align*}
A_{3, \ell} 
= 
\sum_{\substack{ |k_1-\ell| \le 2 \\ k_2 <k_1-3}} 
\sum_{j < k_1-3} 
\sum_{\nu_1 \in \Z^n} 
\sum_{\nu_2 \in \Lambda_{k_2}}
\la T_{\sigma_{j, \K, \Nu}}(f_{1, k_1}, f_{2, k_2}), g_\ell \ra,
\end{align*}
where $\Lambda_{k_2}$ is the same as in \eqref{set-k2}.
By Lemma \ref{lem-Lp1Lp2Lr-est} (1) with $r = p^\prime$ and $N \ge 1$, 
\begin{align*}
 |A_{3, \ell}|
 &\lesssim
 \sum_{\substack{ |k_1-\ell| \le 2 \\ k_2 <k_1-3}} 
 \sum_{j < k_1-3} 
 2^{-(k_1-k_2)n/2-jN} \|f_{1, k_1}\|_{L^{p_1}} \|f_{2, k_2}\|_{L^{p_2}} \|g_\ell\|_{L^{p^\prime}}
 \\
 &\lesssim
 \sum_{\substack{ |k_1-\ell| \le 2 \\ k_2 <k_1-3}} 
 2^{-(k_1-k_2)n/2} \|f_{1, k_1}\|_{L^{p_1}} \|f_{2, k_2}\|_{L^{p_2}} \|g\|_{L^{p^\prime}}
 \\
 &=
 \sum_{|k| \le 2} 2^{kn/2} 
 \sum_{k_2 < \ell - k -3} 
 2^{-(\ell-k_2)n/2} \|f_{1, \ell- k}\|_{L^{p_1}} \|f_{2, k_2}\|_{L^{p_2}} \|g\|_{L^{p^\prime}}.
\end{align*} 
Hence, the left hand side of  \eqref{goal} concerning $A_{3, \ell}$ is estimated by  
\begin{equation*}
\sum_{|k| \le 2} 2^{k(s_1 + n/2)}
\bigg\{ 
\sum_{\ell \ge 0} 
2^{(\ell - k)s_1 q} \|f_{1, \ell- k}\|_{L^{p_1}}^q  
\Big( 
\sum_{k_2 < \ell - k -3} 
2^{\ell s_2- (\ell -k_2)n/2} \|f_{2, k_2}\|_{L^{p_2}} 
\Big)^q 
\bigg\}^{1/q}.
\end{equation*}
It is sufficient to consider the case $k=0$. 
Applying H\"older's inequality, we have
\begin{align} \label{A3-est}
\begin{split}
&\bigg\{ 
\sum_{\ell \ge 0} 
2^{\ell s_1 q} \|f_{1, \ell}\|_{L^{p_1}}^q  
\Big( 
\sum_{k_2 < \ell  -3} 
2^{\ell s_2 -(\ell - k_2)n/2} \|f_{2, k_2}\|_{L^{p_2}} 
\Big)^q 
\bigg\}^{1/q}
\\
&\le
\Big( 
\sum_{\ell \ge 0} 2^{\ell s_1 q_1} 
\|f_{1, \ell}\|_{L^{p_1}}^{q_1}
\Big)^{1/q_1} 
\bigg\{ 
\sum_{\ell \ge 0}  
\Big( 
\sum_{k_2 < \ell  -3} 
2^{\ell s_2 -(\ell - k_2)n/2} \|f_{2, k_2}\|_{L^{p_2}} \Big)^{q_2} 
\bigg\}^{1/q_2}
\\
&\lesssim
\|f_1\|_{B^{s_1}_{p_1, q_1}} 
\bigg\{ 
\sum_{\ell \ge 0} 
\Big( 
\sum_{k_2 < \ell  -3} 
2^{(\ell - k_2)(s_2-n/2)} 2^{k_2s_2}\|f_{2, k_2}\|_{L^{p_2}} 
\Big)^{q_2} 
\bigg\}^{1/q_2}.
\end{split}
\end{align}
If $q_2 \le 1$, then the last sum in \eqref{A3-est} can be estimated by 
\begin{align*}
 \Big( 
 \sum_{\ell \ge 0} 
 \sum_{k_2 < \ell -3} 
 2^{(\ell - k_2)(s_2-n/2)q_2} 
 2^{k_2 s_2 q_2}\|f_{2, k_2}\|_{L^{p_2}}^{q_2} \Big)^{1/q_2}
 &\approx 
  \Big( 
  \sum_{k_2 \ge 0} 
  2^{k_2 s_2 q_2} \|f_{2, k_2}\|_{L^{p_2}}^{q_2} 
  \Big)^{1/q_2}
  \\
 &\lesssim
  \|f_2\|_{B^{s_2}_{p_2, q_2}},
\end{align*} 
where we used the assumption $s_2 < n/2$.
On the other hand, if $q_2 > 1$, then the above sum is estimated by
\begin{align*}
 &\bigg\{ 
 \sum_{\ell \ge 0}  
 \Big( 
 \sum_{k_2 \ge 0} 
 2^{|\ell-k_2|(s_2-n/2)}  2^{k_2s_2}\|f_{2, k_2}\|_{L^{p_2}} 
 \Big)^{q_2} 
 \bigg\}^{1/q_2} 
 \\
 &\le 
  \Big(
  \sum_{\ell \ge 0} 
  2^{|\ell|(s_2-n/2)}
  \Big) 
 \Big( 
 \sum_{k_2 \ge 0}   
 2^{k_2s_2q_2} \|f_{2, k_2}\|_{L^{p_2}}^{q_2} 
 \Big)^{1/q_2}
 \lesssim
 \|f_2\|_{B^{s_2}_{p_2, q_2}},
\end{align*}  
where we used Young's inequality and the assumption $s_2 < n/2$. 
The proof of the assertion (1) is complete.
\end{proof}
%=============================================================================================

%=============================================================================================
\begin{proof}[Proof of the assertion (2).]
Next, we consider the assertion (2). 
We assume that 
 $m_1+m_2=-n/2$, $\sigma \in BS^{(m_1, m_2)}_{0, 0}$, 
$s_2 < m_2 +n/2$ and $s =s_1 +s_2 > -n/2$.
The idea of the proof is similar to the previous one.

\medskip
\noindent
{\it Estimate for $A_{1, \ell}$.}
In the same way as before, 
using the assumption $m_1+m_2 = -n/2$ and Lemma \ref{lem-Lp1Lp2Lr-est} (2) instead of (1), we obtain
\begin{align*}
|A_{1, \ell}| 
&\lesssim  
\sum_{j \ge \ell-6}
\sum_{\substack{k_1 \le j+3 \\ k_2 \le k_1}} 
2^{(k_1- k_2)m_1 -jN} \|f_{1, k_1}\|_{L^{p_1}} \|f_{2, k_2}\|_{L^{p_2}} \|g\|_{L^{p^\prime}}
\\
&\le
\sum_{j \ge \ell-6}
\sum_{k_1 \le j+3}
\sum_{k_2 \le k_1} 
2^{-k_1(s_1-m_1) + k_2(-m_1-s_2)-jN}
\|f_1\|_{B^{s_1}_{p_1, q_1}} \|f_2\|_{B^{s_2}_{p_2, q_2}} \|g\|_{L^{p^\prime}}.
\end{align*}
Since $s_2 < m_2 + n/2 =-m_1$ and $s = s_1 +s_2 > -n/2$, 
the desired estimate follows from the same argument as in the proof of the assertion (1).  
%=============================================================================================

%=============================================================================================
\medskip
\noindent
{\it Estimate for $A_{2, \ell}$.}
By Lemma \ref{lem-Lp1Lp2Lr-est} (2) and  the assumption $m_1+m_2 = -n/2$, 
\begin{align*} 
 |A_{2, \ell}|
 &\lesssim
 \sum_{\substack{k_1 \ge \ell-3 \\ k_1-3 \le k_2 \le k_1}}
 \sum_{j < k_1-3}  
 2^{k_1 m_1 +k_2 m_2 -jN + (j+\ell) n/2} \|f_{1, k_1}\|_{L^{p_1}} \|f_{2, k_2}\|_{L^{p_2}} \|g\|_{L^{p^\prime}}
 \\
 &\approx
 \sum_{0 \le k \le 3} 2^{-k m_2} 
 \sum_{k_1 \ge \ell -3} 
 2^{-(k_1-\ell)n/2} \|f_{1, k_1}\|_{L^{p_1}} \|f_{2, k_1 - k}\|_{L^{p_2}} \|g\|_{L^{p^\prime}}.
\end{align*}
Then, by using the assumption $s > -n/2$, 
we can obtain the desired estimate in the same way as for the assertion (1).
%=============================================================================================

%=============================================================================================
\medskip
\noindent
{\it Estimate for $A_{3, \ell}$.}
It follows from Lemma \ref{lem-Lp1Lp2Lr-est} (2) that
\begin{align*}
|A_{3, \ell}| 
&\lesssim
\sum_{\substack{ |k_1-\ell| \le 2 \\ k_2 <k_1-3}} 
 \sum_{j < k_1-3} 
 2^{k_1 m_1 + k_2m_2 -jN +k_2n/2} \|f_{1, k_1}\|_{L^{p_1}} \|f_{2, k_2}\|_{L^{p_2}} \|g\|_{L^{p^\prime}}
 \\
&\approx
\sum_{|k| \le 2} 
2^{-k m_1} 
\sum_{k_2< \ell-k-3} 
2^{(\ell-k_2)m_1} \|f_{1, \ell-k}\|_{L^{p_1}} \|f_{2, k_2}\|_{L^{p_2}} \|g\|_{L^{p^\prime}}.
\end{align*}
By the same argument as before, it is sufficient to consider the estimate 
\begin{equation*} 
\bigg\{ 
\sum_{\ell \ge 0} 
\Big( 
\sum_{k_2 < \ell -3} 
2^{(\ell -k_2)(s_2+ m_1)} 2^{k_2 s_2}\|f_{2, k_2}\|_{L^{p_2}} 
\Big)^{q_2}
\bigg\}^{1/q_2}
\lesssim 
\|f_2\|_{B^{s_2}_{p_2, q_2}}, 
\end{equation*}
and we see that this estimate holds since $s_2 < -m_1 = m_2 + n/2$.
The proof of Theorem \ref{main1} is complete.
\end{proof}
%=============================================================================================
%=============================================================================================

%=============================================================================================
%=============================================================================================
\section{Proof of Theorem \ref{main2}} \label{section5}
The purpose of this section is to prove Theorem \ref{main2}. 
To obtain Theorem \ref{main2}, it is sufficient to consider the following proposition.

\begin{prop} \label{prop-est-Sobolev}
Let
$1 \le p, p_1, p_2, \widetilde{p}_1, \widetilde{p}_2 \le \infty$ and
 $s, s_1, s_2, \widetilde{s}_1, \widetilde{s}_2 \in \R$.
\begin{enumerate}
\item
Suppose that  the estimate
\begin{equation} \label{est-Sobolev}
\|T_{\sigma}(f_1, f_2)\|_{L^p_s} 
\lesssim 
\|f_1\|_{L^{p_1}_{s_1}}\|f_2\|_{L^{p_2}_{s_2}}
+
\|f_1\|_{L^{\widetilde{p}_1}_{\widetilde{s}_1}}\|f_2\|_{L^{\widetilde{p}_2}_{\widetilde{s}_2}}
\end{equation}
holds for all $\sigma \in BS^{-n/2}_{0, 0}$ and $f_1, f_2 \in \Sh(\R^n)$.
Then, 
$s - \max \{s_i, \widetilde{s}_i \} \le n/2$, $i= 1, 2$, 
and 
$\max \{ s_1 +s_2,\ \widetilde{s}_1 +\widetilde{s}_2 \} \ge -n/2$.
\item
Let $m_1, m_2 \in \R$ be such that $m_1 +m_2 = -n/2$.
Suppose that the estimate \eqref{est-Sobolev} 
holds for all $\sigma \in BS^{(m_1, m_2)}_{0, 0}$ and $f_1, f_2 \in \Sh(\R^n)$. 
Then, 
$s - \max \{ s_i, \widetilde{s}_i \} \le -m_i $, $i=1, 2$, 
and 
$\max \{ s_1 +s_2,\ \widetilde{s}_1 +\widetilde{s}_2 \} \ge -n/2$.
\end{enumerate}
\end{prop}
%=============================================================================================

%=============================================================================================
We shall show that Theorem \ref{main2} follows from Proposition \ref{prop-est-Sobolev}.
Let the exponents of function spaces be the same as in \eqref{pqs-conditions}, 
and let $\sigma \in BS^{-n/2}_{0, 0}$.
Since 
$B^s_{p, q} \hookrightarrow L^p_{s-\epsilon}$ and $ L^{p_i}_{s_i + \epsilon}\hookrightarrow B^{s_i}_{p_i, q_i}$, 
$i=1, 2$, for any $\epsilon >0$, it follows from \eqref{main-est-1} that
\begin{equation*}
\|T_{\sigma}(f_1, f_2)\|_{L^p_{s-\epsilon}}
\lesssim
\|f_1\|_{L^{p_1}_{s_1+\epsilon}}
\|f_2\|_{L^{p_2}_{s_2+ \epsilon}}
+
\|f_1\|_{L^{\widetilde{p}_1}_{\widetilde{s}_1+ \epsilon}}
\|f_2\|_{L^{\widetilde{p}_2}_{\widetilde{s}_2 + \epsilon}}.
\end{equation*}
Therefore, by Proposition \ref{prop-est-Sobolev} (1) 
and the assumption $s=s_1+s_2=\widetilde{s}_1+\widetilde{s}_2$, 
we have 
$s - \max \{ s_i, \widetilde{s}_i\} -2\epsilon \le n/2$ 
and 
$\max\{ s_1+s_2,\ \widetilde{s}_1 + \widetilde{s}_2\} + 2\epsilon \ge -n/2$,
namely, $\min \{ s_i, \widetilde{s}_i\} \le n/2 +2\epsilon$, $i= 1, 2$, and $s +2\epsilon \ge -n/2$.
The arbitrariness of $\epsilon > 0$ implies that 
$\min \{ s_i, \widetilde{s}_i\} \le n/2$, $i= 1, 2$, and $s \ge -n/2$.
Similarly, the assertion (2) of Theorem \ref{main2} follows from that of Proposition \ref{prop-est-Sobolev}.
%=============================================================================================

%=============================================================================================
\medskip
\noindent
{\it Proof of Proposition \ref{prop-est-Sobolev}.}
\ \ 
Let $\phi, \varphi, \psi \in \Sh(\R^n)$ be such that
\begin{align*}
&\supp \phi \subset \{|\xi| \le 2^{1/2}\},
\quad
\supp \varphi \subset \{ |\xi| \le 2\},
\quad
\varphi = 1\ \text{on} \ \{|\xi| \le 2^{1/2}\},
\\
&\supp \psi \subset \{2^{-1/2} \le |\xi| \le 2^{1/2}\},
\quad
\psi = 1 \ \text{on} \ \{2^{-1/4} \le |\xi| \le 2^{1/4}\}.
\end{align*}
%=============================================================================================

%=============================================================================================
Now, we start to prove the assertion (1) in Proposition \ref{prop-est-Sobolev}.
First, we consider the necessity of the condition $\max \{ s_1+s_2,\ \widetilde{s}_1+ \widetilde{s}_2\} \ge -n/2$.
We set 
\begin{align} \label{def-symbol}
\begin{split}
&\sigma(\xi_1, \xi_2) 
= \sum_{k \ge 10} 2^{-kn/2} \varphi(2^{-k}\xi_1) \psi(2^{-k}\xi_2),
\\
&\widehat{f_{1, j}}(\xi_1) = \phi(\xi_1 + 2^{j}e_1), 
\quad
\widehat{f_{2, j}} (\xi_2) = \phi(\xi_2 - 2^j e_1), 
\end{split}
\end{align}
where $j \ge 10$ and $e_1 = (1, 0, \dots, 0) \in \R^n$. 
Since $1+ |\xi_1|+ |\xi_2| \approx 2^k$ 
for $\xi_1 \in \supp \varphi(2^{-k}\cdot)$ and $\xi_2 \in \supp \psi(2^{-k} \cdot)$, 
we see that $\sigma \in BS^{-n/2}_{0, 0}$.
It follows from the support property of $\phi$  that
\begin{equation*}
\supp\phi(\cdot + 2^j e_1)
\subset
\{|\xi + 2^{j}e_1| \le 2^{1/2}\}
\subset
\{2^{j-1/4} \le|\xi| \le 2^{j+1/4}\}.
\end{equation*}
From this and a simple embedding $B^s_{p, 1} \hookrightarrow L^p_s \hookrightarrow B^s_{p, \infty}$ 
(see \cite[Theorem 2.3.8 and Proposition 2.5.7]{Triebel}), we have 
$2^{js_1}
\approx 
\|f_{1, j}\|_{B^{s_1}_{p_1, \infty}}
\le
\|f_{1, j}\|_{L^{p_1}_{s_1}}
\le
\|f_{1, j}\|_{B^{s_1}_{p_1, 1}}
\approx 2^{js_1}, 
$
that is, $\|f_{1, j}\|_{L^{p_1}_{s_1}} \approx 2^{js_1}$. Similarly, we have $\|f_{2, j}\|_{L^{p_2}_{s_2}} \approx 2^{js_2}$.
Moreover, we see that $\varphi(2^{-k}\xi_1)\psi(2^{-k}\xi_2)\phi(\xi_1+ 2^j e_1)\phi(\xi_2 - 2^j e_1)$
 is equal to $\phi(\xi_1+ 2^j e_1)\phi(\xi_2 - 2^j e_1)$ if $k = j$, and 0 otherwise.
Hence, we have
\begin{equation} \label{T-f1j-f2j}
 T_\sigma(f_{1, j}, f_{2, j})(x) 
 =2^{-jn/2} (\F^{-1}\phi(x))^2,
\end{equation}
and consequently $\|T_\sigma(f_{1, j}, f_{2, j})\|_{L^p_s} \approx 2^{-jn/2}$.  
Thus, by our assumption, 
\begin{align*}
 2^{-jn/2}
 \approx
 \|T_\sigma(f_{1, j}, f_{2, j})\|_{L^p_s}
 &\lesssim
 \|f_{1, j}\|_{L^{p_1}_{s_1}}\|f_{2, j}\|_{L^{p_2}_{s_2}}
 +
 \|f_{1, j}\|_{L^{\widetilde{p}_1}_{\widetilde{s}_1}}\|f_{2, j}\|_{L^{\widetilde{p}_2}_{\widetilde{s}_2}}
 \\
 &\approx
 2^{j(s_1+s_2)} + 2^{j(\widetilde{s}_1 + \widetilde{s}_2)}
 \lesssim 
 2^{j\max\{s_1+s_2,\ \widetilde{s}_1 + \widetilde{s}_2 \}}.
\end{align*}
The arbitrariness of $j \ge 10$ implies that $\max\{s_1+s_2,\ \widetilde{s}_1 + \widetilde{s}_2\} \ge -n/2$.

Next, we consider the necessity of the condition $s -\max\{s_2, \widetilde{s}_2 \} \le n/2$. 
We use the same symbol $\sigma$ as in \eqref{def-symbol} and the following functions
\begin{equation} \label{two-functions}
\widehat{f_1}(\xi_1) = \phi(\xi_1),
\quad
\widehat{f_{2, j}}(\xi_2) = \phi(\xi_2 -2^j e_1).
\end{equation}
As mentioned above, 
$\|f_{2, j}\|_{L^{p_2}_{s_2}} \approx 2^{j s_2}$, 
and obviously $\|f_1\|_{L^{p_1}_{s_1}} \approx 1$.
It follows from the same observation as above that  
$ \varphi(2^{-k}\xi_1) \psi(2^{-k}\xi_2)\phi(\xi_1)\phi(\xi_2 - 2^j e_1)$ 
equals $\phi(\xi_1)\phi(\xi_2 - 2^j e_1)$ if $k = j$, and 0 otherwise. 
Hence, we have
\begin{equation*}
T_\sigma (f_1, f_{2, j})(x) = 2^{-j n/2} e^{i 2^j x \cdot e_1} (\F^{-1}\phi(x))^2.
\end{equation*}
Then, 
$\supp \F[T_{\sigma}(f_1, f_{2, j})] 
\subset 
\supp [\phi * \phi](\cdot - 2^j e_1)
\subset
\{2^{j-1/4} \le |\xi| \le 2^{j+1/4}\}$ 
for $j \ge 10$ and consequently $\|T_{\sigma}(f_1, f_{2, j})\|_{L^p_s}  \approx 2^{j(s-n/2)}$.
Therefore, we have
\begin{align*}
2^{j(s-n/2)} 
&\approx
\|T_{\sigma}(f_1, f_{2, j})\|_{L^p_s}
\lesssim
\|f_1\|_{L^{p_1}_{s_1}}\|f_{2, j}\|_{L^{p_2}_{s_2}}
+
\|f_1\|_{L^{\widetilde{p}_1}_{\widetilde{s}_1}}\|f_{2, j}\|_{L^{\widetilde{p}_2}_{\widetilde{s}_2}}
\lesssim 2^{j \max\{s_2, \widetilde{s}_2 \}},
\end{align*}
which holds only when  $s -\max\{s_2, \widetilde{s}_2 \} \le n/2$.
By interchanging the roles of $\xi_1$ and $\xi_2$, we get $s- \max\{s_1, \widetilde{s}_1 \} \le n/2$.
%=============================================================================================

%=============================================================================================
Next, we show the assertion (2). 
Let $m_1, m_2 \in \R$ and $m_1+ m_2 = -n/2$.
We first consider the proof of the assertion $\max\{s_1+s_2, \widetilde{s}_1 + \widetilde{s}_2 \} \ge -n/2$.
Set 
\begin{equation*}
\sigma(\xi_1, \xi_2) 
=
\Big(\sum_{k_1 \ge 10} 2^{m_1 k_1}\psi(2^{-k_1} \xi_1) \Big)
\Big(\sum_{k_2 \ge 10} 2^{m_2k_2}\psi(2^{-k_2} \xi_2) \Big).
\end{equation*}
It is not difficult to check that $\sigma \in BS^{(m_1, m_2)}_{0, 0}$
since $|\xi_i| \approx 2^{k_i}$ if $\xi_i \in \supp \psi (2^{-k_i} \cdot)$, 
$i=1, 2$. 
Using this symbol and the functions in \eqref{def-symbol}, 
we obtain the same representation as in \eqref{T-f1j-f2j} 
since $\psi(2^{-k_1} \xi_1)\psi(2^{-k_2} \xi_2) \phi(\xi_1 + 2^j e_1) \phi(\xi_2 -2^j e_1)$
is equal to
$\phi(\xi_1 + 2^j e_1) \phi(\xi_2 -2^j e_1)$ 
if $k_1=k_2=j$, and 0 otherwise. The rest of the proof is the same as for the case $\sigma \in BS^{-n/2}_{0, 0}$.

We next prove that  $s- \max \{s_i, \widetilde{s}_i\} \le -m_i$, $i= 1, 2$.
We use  the functions in \eqref{two-functions} and the following symbol
\begin{equation*}
\sigma(\xi_1, \xi_2) 
=
\varphi(\xi_1)
\Big(\sum_{k_2 \ge 10} 2^{m_2k_2}\psi(2^{-k_2} \xi_2) \Big).
\end{equation*}
Note that $\sigma \in BS^{(m_1, m_2)}_{0, 0}$.
Since 
$\varphi(\xi_1)\psi(2^{-k_2} \xi_2)\phi(\xi_1)\phi(\xi_2 -2^j e_1)$ 
equals 
$\phi(\xi_1)\phi(\xi_2 -2^j e_1)$ 
if $k_2=j$, and 0 otherwise, we have 
\begin{equation*}
T_{\sigma}(f_1, f_{2, j})(x)
= 2^{j m_2}e^{i 2^j x \cdot e_1}  ( \F^{-1}\phi(x) )^2, 
\end{equation*}
and hence 
$
\|T_{\sigma}(f_1, f_{2, j})\|_{L^p_s}  \approx 2^{j(s+m_2)}
$.
Therefore, by the same argument as above,  
we conclude that $s- \max \{s_2, \widetilde{s}_2\} \le -m_2$.
Applying the argument with the variables $\xi_1$ and $\xi_2$ interchanged, 
we obtain $s - \max \{s_1, \widetilde{s}_1 \} \le -m_1$.
The proof  is complete.
%=============================================================================================

%=============================================================================================
\begin{rem} \label{rem-Sobolev}
Using the same functions as in \eqref{two-functions},
we can prove that \eqref{Kato-Ponce-modoki} holds only if  $s \le \max\{s_2, \widetilde{s}_2 \}$, 
namely, $\min \{s_1, \widetilde{s}_1\} \le 0$. 
Similarly, 
$\min\{s_2, \widetilde{s}_2 \} \le$ $0$
holds.
\end{rem}
%=============================================================================================
%=============================================================================================
\section*{Acknowledgement}
The author would like to thank Professor Naohito Tomita 
for a lot of helpful suggestions and warm encouragement.
%=============================================================================================
%=============================================================================================

%=============================================================================================
%=============================================================================================

\end{document}